\documentclass[english]{amsart}
\usepackage{dsfont}
\usepackage{url}
\usepackage[utf8]{inputenc}
\usepackage[T1]{fontenc}
\usepackage{lmodern}
\usepackage{babel}
\usepackage{mathtools}  
\usepackage{amssymb}
\usepackage{lipsum}
\usepackage{mathrsfs}
 \usepackage{mathabx}
\usepackage{color}
\usepackage{ulem}

\newtheorem{theorem}{Theorem}
\newtheorem{proposition}{Proposition}
\newtheorem{corollary}{Corollary}

\newtheorem{remark}{Remark}

\title[Elliptic Cauchy problem]{New global logarithmic stability of the Cauchy problem for elliptic equations}

\author[Mourad Choulli]{Mourad Choulli}
\address{Universit\'e de Lorraine, 34 cours L\'eopold, 54052 Nancy cedex, France}
\email{mourad.choulli@univ-lorraine.fr}

\thanks{The author is supported by the grant ANR-17-CE40-0029 of the French National Research Agency ANR (project MultiOnde). }

\begin{document}

\begin{abstract}
In this  short paper we prove a global logarithmic stability of the Cauchy problem for $H^2$-solutions of an anisotropic elliptic equation in a Lipschitz domain. The result we obtained is based on tools borrowed from the existing technics to establish stability estimate for the Cauchy problem \cite{Ch1} (see also \cite{BC}) combined with tools we already used in \cite{CT} to study an inverse medium problem.
\end{abstract}



\maketitle


Throughout this text, $\Omega$ is a Lipschitz bounded domain of $\mathbb{R}^n$, $n\ge 2$, and $\Gamma$ is a nonempty open subset of $\partial \Omega$. Consider then the divergence form elliptic operator $L$ that acts as follows
\[
Lu(x)=\mbox{div}(A(x)\nabla u(x) ),
\]
where $A=(a^{ij})$ is a symmetric matrix, with  coefficients in $W^{1,\infty}(\Omega )$, so that there exist $\kappa >0$ and $\lambda \ge 1$ for which 
\begin{equation}\label{1}
\lambda ^{-1}|\xi |^2 \le A(x)\xi \cdot \xi \le \lambda |\xi |^2,\quad  x\in \Omega , \; \xi \in \mathbb{R}^n,
\end{equation}
and
\begin{equation}\label{2}
\sum_{k=1}^n\left|\sum_{i,j=1}^n\partial_k a^{ij}(x)\xi _i\xi_j\right| \le \kappa |\xi|^2,\quad x\in \Omega ,\; \xi \in \mathbb{R}^n.
\end{equation}

The Cauchy problem we consider can be stated as follows: given $(F,f,g)\in L^2(\Omega )\times L^2(\Gamma )\times L^2(\Gamma )^n$, find $u\in H^2(\Omega )$ obeying to the boundary value problem
\begin{equation}\label{cp}
\left\{
\begin{array}{ll}
Lu(x)=F(x)\quad &\mbox{a.e. in}\; \Omega .
\\
u(x)=f &\mbox{a.e. on}\; \Gamma ,
\\
\nabla u(x)=g &\mbox{a.e. on}\; \Gamma .
\end{array}
\right.
\end{equation}
It is well known that this problem may not have a solution and, according to the classical uniqueness of continuation from Cauchy data, the boundary value problem \eqref{cp} has at most one solution. Moreover, even if the solution of \eqref{cp} exists, the continuous dependence of the solution on the data $(F,f,g)$ is not in general Lipschitz. In other words, the Cauchy problem is ill-posed in Hadamard's sense. As it is shown by Hadamard \cite{Ha}, the modulus of continuity of the mapping $(F,f,g)\mapsto u$ can be of logarithmic type. Therefore, for the general Cauchy problem, the logarithmic type stability estimate is the best possible one that we can expect.

We aim here to prove  the following result.
\begin{theorem}\label{theorem1}
Let $0<s<\frac{1}{2}$. Then there exist two constants $c>0$ and $C>0$, only depending on $s$, $\Omega$, $\Gamma$, $\lambda$ and $\kappa$, and $\delta_0$ only depending  on $\Omega$, so that, for any $u\in H^2(\Omega )$, $0<\delta<\delta_0$ and $j=0,1$, we have
\begin{align}
C\|u\|_{H^j(\Omega )}&\le \delta ^{\frac{s}{j+1}}\|u\|_{H^{j+1}(\Omega )}\label{0}
\\
&+e^{e^{c/\delta}}\left(\|u\|_{L^2(\Gamma)}+\|\nabla u\|_{L^2(\Gamma)}+\|Lu\|_{L^2(\Omega)}\right).\nonumber
\end{align}
\end{theorem}

As usual, the interpolation inequality \eqref{0} yields a double logarithmic stability estimate. Precisely, we have the following corollary in which
\[
\Psi _{s,j}^c(\rho) =\left\{ \begin{array}{ll} \left(\ln \ln \rho \right)^{-\frac{s}{j+1}}\quad &\mbox{if}\; \rho >c, \\ \rho &\mbox{if}\; 0<\rho <c,\end{array}\right. \quad j=0,1,
\]
extended by continuity at $\rho=0$ by setting $\Psi^c _{s,j}(0)=0$, where $c>e$.

\begin{corollary}\label{corollary1}
Let $0<s<\frac{1}{2}$. Then there exist two constants $c>e$ and $C>0$, only depending on $s$, $\Omega$, $\Gamma$, $\lambda$ and $\kappa$ so that, for any $u\in H^2(\Omega )$, $u\ne 0$ and $j=0,1$, we have
\[
C\|u\|_{H^j(\Omega )}\le \|u\|_{H^{j+1}(\Omega )} \Psi^c_{s,j}\left( \frac{\|u\|_{H^{j+1}(\Omega )}}{\|u\|_{L^2(\Gamma)}+\|\nabla u\|_{L^2(\Gamma)}+\|Lu\|_{L^2(\Omega)}}\right).
\]
\end{corollary}

As we observed above, according to the classical uniqueness of continuation from Cauchy data for elliptic equations, if $u\in H^2(\Omega )$ satisfies $Lu=0$ in $\Omega$, $u=0$ on $\Gamma$ and $\nabla u=0$ on $\Gamma$, then $u=0$.

To our knowledge the optimal stability estimate for the Cauchy problem for an elliptic equation hold in two cases : (i) Lipschitz domain and $C^{1,\alpha}$-solutions  and (ii) $C^{1,1}$ domain and $H^2$-solutions. This optimal stability estimate is of single logarithmic type. For the case (i), we refer to \cite{Ch1} under an additional geometric condition on the domain. This condition was removed in \cite{BC} (see also \cite{Ch2}). A similar result was obtained in \cite{BD} for the Laplace operator. The case (ii) was established in \cite{B} for the Laplace operator. However the results in \cite{B} can be extended to an anisotropic elliptic operator in divergence form. In the present paper we deal with the case of Lipschitz domain and $H^2$-solutions. For this case we are only able to get a stability estimate of double logarithmic type (Corollary \ref{corollary1}). We do not know whether this result can be improved to a single logarithmic type.

Let us explain briefly the main steps to obtain the global stability estimate for the Cauchy problem. The first step consists in continuing a well chosen interior data to the boundary. In the second step we continue the data from an interior subdomain to another subdomain. The continuation of the Cauchy data to some interior subdomain constitute the third step. For the last two steps it is sufficient to assume that the domain is Lipschitz and the solutions have $H^2$-regularity. While in the first step, it is necessary to assume that either the domain is $C^{1,1}$ or the solutions have $C^{1,\alpha}$-regularity. Apart from these two cases we do not know how to prove the continuation result in the first step. It is worth mentioning that the last two steps give rise to a stability estimate of H\"older type and for the first step the stability estimate we obtain is of logarithmic type.

Since we can not use this classical scheme to prove Theorem \ref{theorem1}, we modify it slightly to avoid the use of the first step. The main idea consists in refining the second step. Precisely, we show that we can continue the data, away from the boundary, from a ball with arbitrary small radius to another ball with the same radius, with an exact dependence of the constants on the radius. This new step yields a stability estimate of double logarithmic type. It turn out that this result is optimal if one uses to prove it three-ball inequalities. For this reason we think that the actual technics, based on three-ball inequalities, can not be used to improve Theorem \ref{theorem1}.

As we already mentioned, the proof of Theorem \ref{theorem1} consists in an adaptation of existing results. The proposition hereafter is proved in \cite{Ch1} under an additional geometric condition and for a Lipschitz domain in \cite{BC} (see also \cite{Ch2}).

Henceforward, $C_0$ is a generic constant only depending on $\Omega$, $\lambda$ and $\kappa$, while $C_1$ is a generic constant only depending $\Omega$, $\Gamma$, $\lambda$ and $\kappa$.

\begin{proposition}\label{proposition1}
There exist a constant $\gamma >0$ and a ball $B$ in $\mathbb{R}^n$ satisfying $B\cap \Omega \ne \emptyset$, $B\cap (\mathbb{R}^n\setminus\overline{\Omega})\ne\emptyset$ and $B\cap \partial \Omega\Subset \Gamma$, only depending on $\Omega$, $\Gamma$, $\lambda$ and $\kappa$,  so that, for any  $u\in H^2(\Omega )$ and $\epsilon >0$, we have
\begin{equation}\label{3}
C_1\|u\|_{H^1(B\cap \Omega  )}\le \epsilon ^\gamma \|u\|_{H^1(\Omega )} +\epsilon ^{-1}\left(\|u\|_{L^2(\Gamma)}+\|\nabla u\|_{L^2(\Gamma)}+\|Lu\|_{L^2(\Omega)}\right).
\end{equation}
\end{proposition}

\begin{proof}[Proof of Theorem \ref{theorem1}]
Let $B$ as in the preceding proposition. Pick then $\tilde{x}\in B\cap \partial \Omega$. As $B\cap \Omega$ is Lipschitz, it contains a cone with vertex at $\tilde{x}$. That is we can find $R>0$, $\theta \in ]0,\pi /2[$ and $\xi \in \mathbb{S}^{n-1}$ so that
\[
\mathcal{C}(\tilde{x})=\left\{x\in \mathbb{R}^n;\; 0<|x-\tilde{x}|<R,\; (x-\tilde{x})\cdot \xi >|x-\tilde{x}|\cos \theta \right\}\subset B\cap \Omega .
\]
Let $x_\delta = \tilde{x}+ \frac{\delta}{3\sin \theta} \xi$, with $\delta < \frac{3R\sin \theta}{2}$. Then $\mbox{dist}(x_\delta ,\partial (B\cap \Omega))> 3\delta $.

Define, for $\delta >0$,
\begin{align*}
&\Omega^\delta =\{x\in \Omega ;\; \mbox{dist}(x,\partial \Omega )>\delta\},
\\
&\Omega_\delta =\{x\in \Omega ;\; \mbox{dist}(x,\partial \Omega )<\delta\}
\end{align*}
and set
\[
\delta ^\ast =\sup \{\delta >0;\; \Omega ^\delta \ne \emptyset\}.
\]

Let $0<\delta \le \delta ^\ast/3$. Then a slight modification of the proof of \cite[Theorem 2.1, step 1]{CT} yields, for any $u\in H^2(\Omega )$, $y,y_0\in \Omega^{3\delta}$ and $\epsilon >0$,
\begin{equation}\label{4}
C_0\|u\|_{L^2(B(y,\delta ))}\le  \epsilon^{\frac{1}{1-\psi (\delta )}}\| u\|_{L^2(\Omega )}+\epsilon^{-\frac{1}{\psi (\delta )}}\left(\|Lu\|_{L^2(\Omega )}+\|u\|_{L^2(B(y_0,\delta ))}\right).
\end{equation}

Here $\psi$ is of the form $\psi(\delta )=se^{-C_0/\delta}$, with $0<s<1$ only depending on $\Omega$, $\lambda$ and $\kappa$.

Putting together \eqref{3} and \eqref{4} with $y_0=x_\delta$, we find, for any $u\in H^2(\Omega )$, $y\in \Omega^{3\delta}$, $0<\delta < \delta_0:=\min \left(\frac{\delta ^\ast}{3},\frac{3R\sin \theta}{2}\right)$, $\epsilon >0$ and $\eta >0$,
\begin{align}
&C_1\|u\|_{L^2(B(y,\delta ))}\le  \epsilon^{\frac{1}{1-\psi (\delta )}}\| u\|_{L^2(\Omega )}\label{5}
\\
&+\epsilon^{-\frac{1}{\psi (\delta )}}\left[\|Lu\|_{L^2(\Omega )} +\eta ^\gamma \|u\|_{H^1(\Omega )} +\eta ^{-1}\left(\|u\|_{L^2(\Gamma)}+\|\nabla u\|_{L^2(\Gamma)}+\|Lu\|_{L^2(\Omega)}\right)\right].\nonumber
\end{align}

In \eqref{5}, take 
\[
\eta =\epsilon^{\frac{1}{\gamma \psi (\delta )(1-\psi (\delta ))}}
\]
in order to obtain 
\begin{align}
C_1\|u\|_{L^2(B(y,\delta ))}&\le \phi _0(\epsilon ,\delta)\|u\|_{H^1(\Omega )}\label{6}
\\
&+\phi_1(\epsilon ,\delta)\left(\|u\|_{L^2(\Gamma)}+\|\nabla u\|_{L^2(\Gamma)}+\|Lu\|_{L^2(\Omega)}\right),\nonumber
\end{align}
where
\begin{align*}
&\phi _0(\epsilon ,\delta)=\epsilon^{\frac{1}{1-\psi (\delta )}},
\\
&\phi_1(\epsilon ,\delta)=\epsilon^{-\frac{1}{\psi (\delta )}}\max \left(1, \epsilon^{-\frac{1}{\gamma \psi (\delta )(1-\psi (\delta ))}}\right).
\end{align*}

On the other hand, it is straightforward to check that $\Omega^{3\delta}$ can be recovered by at most $k^n$ balls with center in $\Omega^{3\delta}$ and radius $\delta$, where $k=[c/\delta]$, the constant $c$ only depends on $n$ and the diameter of $\Omega$. Whence we obtain from \eqref{6}
\begin{align}
C_1\|u\|_{L^2(\Omega^{3\delta})}&\le \delta^{-n}\phi _0(\epsilon ,\delta)\|u\|_{H^1(\Omega )}\label{7}
\\
&+\delta^{-n}\phi_1(\epsilon ,\delta)\left(\|u\|_{L^2(\Gamma)}+\|\nabla u\|_{L^2(\Gamma)}+\|Lu\|_{L^2(\Omega)}\right).\nonumber
\end{align}

Now, according to Hardy's inequality (see for instance \cite[Theorem 1.4.4.4, page 29]{Gr}), for $0<s<\frac{1}{2}$, there exists $\varkappa$, only depending on $\Omega$ and $s$ so that
\[
\|u\|_{L^2(\Omega _{3\delta})}\le (3\delta)^s \left\|\frac{u}{\mbox{dist}(x,\partial \Omega)^s}\right\|_{L^2(\Omega )}\le \varkappa \delta ^s\|u\|_{H^s(\Omega )}.
\]
As $H^1(\Omega )$ is continuously embedded in $H^s(\Omega)$, changing if necessary $\varkappa$, we have
\begin{equation}\label{8}
\|u\|_{L^2(\Omega _{3\delta})}\le  \varkappa \delta ^s\|u\|_{H^1(\Omega )}.
\end{equation}

Henceforward, $0<s<\frac{1}{2}$ is fixed and $C$ is a generic constant that only depends on $\Omega$, $\Gamma$, $\lambda$, $\kappa$ and $s$.

Putting together \eqref{7} and \eqref{8}, we get
\begin{align}
C\|u\|_{L^2(\Omega )}&\le \left(\delta^{-n}\phi _0(\epsilon ,\delta)+\delta ^s\right)\|u\|_{H^1(\Omega )}\label{9}
\\
&+\delta^{-n}\phi_1(\epsilon ,\delta)\left(\|u\|_{L^2(\Gamma)}+\|\nabla u\|_{L^2(\Gamma)}+\|Lu\|_{L^2(\Omega)}\right).\nonumber
\end{align}

We take in \eqref{9} $\epsilon$ so that $\delta^{-n}\phi _0(\epsilon ,\delta)=\delta ^s$ or equivalently $\epsilon =\delta^{(n+s)\psi (\delta )}$. Then elementary computations yield
\begin{align}
C\|u\|_{L^2(\Omega )}&\le \delta ^s\|u\|_{H^1(\Omega )}\label{10}
\\
&+e^{e^{c/\delta}}\left(\|u\|_{L^2(\Gamma)}+\|\nabla u\|_{L^2(\Gamma)}+\|Lu\|_{L^2(\Omega)}\right).\nonumber
\end{align}
This inequality corresponds to \eqref{0} when $j=0$.

Next, noting that $H^1(\Omega )$ can be seen as an interpolated space between $L^2(\Omega )$ and $H^2(\Omega )$, we have pour $\epsilon >0$,
\[
C_\Omega \|u\|_{H^1(\Omega )}\le \epsilon \|u\|_{H^2(\Omega )}+\epsilon^{-1}\|u\|_{L^2(\Omega )},
\]
the constant $C_\Omega$ only depends on $\Omega$.

This last inequality with $\epsilon =\delta^{s/2}$ and \eqref{10} entail
\begin{align*}
C\|u\|_{H^1(\Omega )}&\le \delta ^{s/2}\|u\|_{H^2(\Omega )}
\\
&+e^{e^{c/\delta}}\left(\|u\|_{L^2(\Gamma)}+\|\nabla u\|_{L^2(\Gamma)}+\|Lu\|_{L^2(\Omega)}\right).
\end{align*}
That is we proved \eqref{0} in the case $j=1$. 
\end{proof}

\begin{remark}
{\rm
(i) It is worth mentioning that Theorem \ref{theorem1} still holds if $L$ is substituted by $L+L_1$, where $L_1$ is a first order partial differential operator with bounded coefficients. In that case the constants $c$ and $C$ in the statement of Theorem \ref{theorem1} may also depend on bounds on the coefficients of $L_1$.
\\
(ii) For $0<t<2$, we have the following interpolation inequality, with $u\in H^2(\Omega)$ and $\epsilon >0$,
\[
C_\Omega \|u\|_{H^t(\Omega )}\le \epsilon ^{\frac{t}{2-t}}\|u\|_{H^2(\Omega )}+\epsilon^{-1}\|u\|_{L^2(\Omega )}.
\]
We can then proceed as in the preceding proof in order to get, for $u\in H^2(\Omega )$ and $0<\delta <\delta_0$,
\begin{align*}
C\|u\|_{H^t(\Omega )}&\le \delta ^{\frac{st}{2}}\|u\|_{H^2(\Omega )}
\\
&+e^{e^{c/\delta}}\left(\|u\|_{L^2(\Gamma)}+\|\nabla u\|_{L^2(\Gamma)}+\|Lu\|_{L^2(\Omega)}\right).
\end{align*}
Here the constants $c$ and $C$ only depend on $s$, $t$, $\Omega$, $\Gamma$, $\lambda$ and $\kappa$, and $\delta_0$ only depends on $\Omega$.
}
\end{remark}

\end{document}